\documentclass[11pt]{amsart}
\usepackage{preamble}

\author{Noam Kimmel}
\title{mass equidistribution for {P}oincar\'e series of large index}
\address{N. Kimmel: Raymond and Beverly Sackler School of Mathematical Sciences, Tel Aviv University, Tel Aviv 69978, Israel.}
\email{\href{mailto:noamkimmel@mail.tau.ac.il}{noamkimmel@mail.tau.ac.il}}
\thanks{This research was supported by the European Research Council (ERC) under the European Union's  Horizon 2020 research and innovation program  (Grant agreement No.    786758).}

\begin{document}

\maketitle

\begin{abstract}
Let $P_{k,m}$ denote the Poincar\'e series of weight $k$ and index $m$ for the full modular group $\mathrm{SL}_2(\mathbb{Z})$, and let $\{P_{k,m}\}$ be a sequence of Poincar\'e series for which $m(k)$ satisfies $m(k) / k \rightarrow\infty$ and $m(k) \ll k^{\frac{3}{2} - \epsilon}$.
We prove that the $L^2$ mass of such a sequence equidistributes on $\mathrm{SL}_2(\mathbb{Z}) \backslash \mathbb{H}$ with respect to the hyperbolic measure as $k$ goes to infinity.
As a consequence, we deduce that the zeros of such a sequence $\{P_{k,m}\}$ become uniformly distributed in $\mathrm{SL}_2(\mathbb{Z}) \backslash \mathbb{H}$ with respect to the hyperbolic measure.
\end{abstract}

\setcounter{tocdepth}{2}
{\tableofcontents}

\section{Introduction}

\subsection{Mass equidistribution}

Let $(M,g)$ be a Riemannian manifold with an associated Laplace-Beltrami operator $\Delta_M$.
A central problem in Quantum chaos is to understand the limiting behaviour of eigenfunctions $\psi_\lambda\in L^2(M,\mathrm{dvol}_g)$ of $-\Delta_M$ as the associated eigenvalue $\lambda$ grows.

An important instance of this problem which attracted a lot of attention is the case of Maass cusp forms on the hyperbolic surface $X = \mathrm{SL}_2(\ZZ) \backslash \HH$.
Let $\phi$ denote such a Maass form with corresponding eigenvalue $\lambda$, and assume $\phi$ is normalized so that $\int_X |\phi|^2 \mathrm{d}\mu = 1$ where $\mathrm{d}\mu$ is the hyperbolic measure $\mathrm{d}\mu = \frac{\mathrm{d}x\mathrm{d}y}{y^2}$.
In this setting, Zelditch \cite{MR1105653} has shown that as $\lambda \rightarrow\infty$, for a typical Maass form $\phi$, the measure $\mathrm{d}\mu_\phi = |\phi(x+iy)|^2 \frac{\mathrm{d}x\mathrm{d}y}{y^2}$ converges weakly to the uniform distribution measure $\frac{3}{\pi}\mathrm{d}\mu$ on $X$.
This statement is referred to as Quantum Ergodicity.
For negative curvature surfaces such as $X$, it was conjectured by Rudnick and Sarnak \cite{MR1266075} that a stronger result holds - that for every Maass cusp form $\phi$ the distribution $\mathrm{d}\mu_\phi$ converges weakly to the uniform distribution measure, without the necessity of excluding a negligible amount of Maass cusp forms.
This statement is referred to as Quantum Unique Ergodicity (QUE).
For the surface $X$, QUE for Hecke Maass forms was proven in the works of Lindenstrauss \cite{MR2195133} and Soundararajan \cite{MR2680500}. 
For further material related to QUE we refer the reader to \cite{MR1361757, MR2757360, MR2985318}. 

A well studied variant of QUE on $X$ is to consider holomorphic modular forms in place on Maass cusp forms.
In this variant, the role of the Laplacian is replaced with the Cauchy-Riemann equations.
And so, instead of considering a sequence of Maass cusp forms $\phi_\lambda$ associated to increasing eigenvalues $\lambda$, one considers a sequence of holomorphic cusp forms $f_k$ of increasing weight $k$, normalized so that $\int_{X}y^k |f_k(x+iy)|^2 \mathrm{d}\mu = 1$.
One then wishes to show that the corresponding measure $\mathrm{d}\mu_f = y^k |f_k(x+iy)|^2 \frac{\mathrm{d}x \mathrm{d}y}{y^2}$ converges weakly to the uniform measure on $X$.
A sequence of holomorphic modular forms $\{f_k\}$ which, after normalizing each element in the sequence, satisfies this condition, is also referred to as having mass equidistribution.

However, mass equidistribution cannot hold for all sequence $\{f_k\}$ of holomorphic modular forms.
Indeed, one can find sequences of holomorphic cusp forms of increasing weight that, after normalization, do not approach the uniform measure (such as $(\Delta^n)_{n\geq 1}$, where $\Delta(z)$ is the Ramanujan delta function).
The failure stems from the fact that $\SS_k$, the space of holomorphic cusp forms of weight $k$ for $\mathrm{SL}_2(\ZZ)$, has a dimension of roughly $\frac{k}{12}$ (as a linear space over $\CC$), as opposed to the case of Maass cusp forms where each eigenspace of $-\Delta_X$ is expected to be very small.
And so, when considering the holomorphic analogue of QUE, one has to find a formulation that deals with this dimension issue.
This is usually done by restricting attention to special holomorphic modular forms within each $\SS_k$.

The most natural restriction is to consider Hecke cusp forms, those elements in $\SS_k$ which are eigenforms of all Hecke operators.
For such forms, it was shown by Holowinsky and Soundararajan \cite{MR2680499} that holomorphic QUE holds.
This result has also been generalized in many works (see for example \cite{MR2852367, MR2886501,MR2813338, MR4057145, MR3807308}).

However, it is interesting to consider whether there are other families of cusp forms besides Hecke eigenforms which also satisfy holomorphic QUE.
In \cite{huang2022effective}, Huang showed that small linear combinations of Hecke eigenforms also satisfy holomorphic QUE, which provides an example of such a family.

In this paper we present another family of special cusp forms which we show also satisfy holomorphic QUE.
Namely, we are going to consider Poincar\'e series of large index.
The precise statement of mass equidistribution we prove is given in \autoref{thm-mass_equi}.

\subsection{Zeros of Poincar\'e series}
The motivation for this project came from exploring the zeros of Poincar\'e series.
There is a remarkable connection between mass equidistribution for a sequence $\{f_k\}$ and the asymptotic distribution of its zeros.
This connection was found by Rudnick in \cite{MR2181743} where he showed that mass equidistribution for a sequence of cusp forms $\{f_k\}$ also implies that their zeros become uniformly distributed in $\SL\backslash\HH$ with respect to the hyperbolic measure.

Regarding the zeros of Poincar\'e series, Rankin showed in \cite{MR0664646} that at least $\frac{k}{12} - m$ of the $\frac{k}{12} + \BigO{1}$ zeros of $P_{k,m}$ in the standard fundamental domain $\mathcal{F}$ lie on the arc $|z| = 1$ (where $P_{k,m}$ is the Poincar\'e series of weight $k$ and index $m$ for the full modular group).
This shows that for any fixed $m\geq 1$, the mass of the sequence $\{P_{k,m}\}$ does not equidistribute, since the zeros certainly do not become uniformly distributed in $\SL\backslash\HH$ as $k$ grows.
However, this result does not exclude sequences where $m$ grows with $k$.

In a recent paper \cite{kimmel2024asymptotic}, the author showed that for a sequence of Poincar\'e series $\{P_{k,m}\}$ with $m(k) \sim \alpha k$ for some constant $\alpha$, many of the zeros of $\{P_{k,m}\}$ tend (as $k$ grows) towards specific curves which depend on $\alpha$.
And so, once again, one can deduce that the mass of such a sequence cannot equidistribute.
However, the curves above become increasingly more intricate and dense in $\mathcal{F}$ as $\alpha$ increases.
This might suggest that the zeros of a sequence $\{P_{k,m}\}$ become uniformly distributed in $\mathcal{F}$ if $m(k)/k\rightarrow \infty$.
Due to Rudnick's result that mass equdistribution implies uniform distribution of the zeros, this naturally raises the question if an even stronger condition holds:
\begin{question}
Does the sequence $\{P_{k,m}\}$ have mass equidistribution when $m(k)$ is such that ${m(k)}/{k}\rightarrow \infty$?
\end{question}

\subsection{Main results}
Our main result is an affirmative answer, under certain conditions, to the question above.

We remind the reader of the definition of the Poincar\'e series of weight $k$ and index $m$ for the full modular group.
This is given by
\begin{equation}\label{eq-pkm_def}
P_{k,m}(z)=
\sum_{\gamma \in \Gamma_\infty \backslash \mathrm{SL}_2(\ZZ)}
e(m\gamma z) j(\gamma, z)^{-k}
= \frac{1}{2} \sum_{\substack{c,d\in\ZZ \\ \gcd(c,d) = 1}}
\frac{\exp{2\pi i m \gamma_{c,d}  z}}{(cz + d)^k} 
\end{equation}
where 
$$
\Gamma_\infty = \left\lbrace \pm\begin{pmatrix}
    1       & *  \\
    0       & 1 
\end{pmatrix} \in \mathrm{SL}_2(\ZZ) \right\rbrace, \quad
j\left(
\begin{pmatrix}
    a       & b  \\
    c       & d 
\end{pmatrix}
,z \right) = cz+d,
$$
$e(z) = e^{2\pi i z}$, and $\gamma_{c,d}$ is any matrix in $\mathrm{SL}_2(\ZZ)$ with bottom row $(c,d)$.

With these notation we have:
\begin{theorem}\label{thm-mass_equi}
Let $\epsilon > 0$.
For any $m(k)$ satisfying
\begin{equation*}
k
\lll m \ll
k^{\frac{3}{2} - \epsilon}
\end{equation*}
the sequence $\left\lbrace P_{k,m}\right\rbrace$ has mass equidistribution.
That is, for every smooth and compactly supported function $\psi\in C_c^\infty\left(\SL\backslash\HH\right)$ we have
\begin{equation}
\int_{\SL\backslash\HH}\left|\widetilde{P}_{k,m}(z)\right|^2 \psi(z) y^k \mathrm d \mu
\\
\xrightarrow{k\rightarrow\infty}
\frac{3}{\pi}\int_{\SL\backslash\HH}\psi(z) \mathrm d \mu.
\end{equation}
where 
$$
\widetilde{P}_{k,m} = \frac{P_{k,m}}{\BRAK{P_{k,m}, P_{k,m}}^{\frac{1}{2}}}.
$$
\end{theorem}
Here we used the notation $k\lll m$ to indicate $m(k)/k \rightarrow \infty$.

\begin{remarks}
\text{}
\begin{enumerate}
    \item Later, in \eqref{eq-pkm_normalization}, we show that for $m(k)$ as above the normalization factor satisfies 
    $$
    \BRAK{P_{k,m},P_{k,m}} \sim \frac{\Gamma\left(k-1\right)}{\left(4\pi m \right)^{k-1}}.
    $$
    \item The condition $m \ggg k$ is necessary, since otherwise the zeros of $\{P_{k,m}\}$ do not equidistribute \cite{kimmel2024asymptotic}.
    \item Our methods of proof work in the range $k\lll m\ll k^{3/2 -\epsilon}$, but it is an interesting question if mass equidistribution continues for $m\gg k^{3/2}$.
    We note however that tackling the range $m\ggg k^2$ seems out of reach of current methods, since it is not even known whether or not $P_{k,m}$  identically vanishes in this range.
\end{enumerate}
\end{remarks}

As we remarked earlier, \autoref{thm-mass_equi} also implies uniform distribution of zeros.
This will be discussed further in \autoref{subsec-zeros}, in which we show:
\begin{corollary}\label{cor-zeros}
Let $\epsilon > 0$.
For any $m(k)$ satisfying
\begin{equation*}
k
\lll m \ll
k^{\frac{3}{2} - \epsilon}
\end{equation*}
the zeros of the sequence $\left\lbrace P_{k,m}\right\rbrace$ equidistribute in $\mathrm{SL}_2(\mathbb{Z}) \backslash \mathbb{H}$ (in the sense given in \autoref{thm-equidist_zeros}).
\end{corollary}

\subsection{Proof plan and paper structure}
Previous results regarding holomorphic QUE were always related to Hecke eigenforms, for which it is a well known theme that mass equidistribution is equivalent to subconvexity estimates of related $L$ functions.
However, in this paper we consider mass equidistribution of Poincar\'e series, which unlike previous examples are not related to Hecke eigenforms in any obvious exploitable way.
And so it makes sense that different methods must be used.

In order to prove mass equidistribution we begin as usual by considering the inner product of $|P_{k,m}|^2$ against functions in the spectral decomposition of $\Delta$ on $L^2\left(\SL \backslash \HH\right)$.
We then apply the Rankin-Selberg unfolding trick, which lends itself nicely due to the definition of $P_{k,m}$ as an average over $\Gamma_\infty \backslash\SL$.
The problem then reduces to various estimates of Fourier coefficients of $P_{k,m}$, which we require with uniformity in the weight $k$.

And so, \autoref{sec-pkm_coef} is dedicated to the study of the Fourier coefficients of Poincar\'e series.
We give upper bounds for the $n$-th Fourier coefficient of $P_{k,m}$, uniform in $m,n,k$, and we also give asymptotics for the $m$-th Fourier coefficient of $P_{k,m}$ under some restrictions.

Using these results, in \autoref{sec-mass}, we prove \autoref{thm-mass_equi} regarding the mass equidistribution of Poincar\'e series.
Afterwards, in \autoref{subsec-zeros}, we discuss the distribution of zeros and prove \autoref{cor-zeros}.

\subsection{Notations for modular forms}
We now present some notations and basic facts about modular forms. 
Throughout the paper, $k\in \NN$ will always be an even positive integer.
We denote by $\MM_k$ the space of modular forms of weight $k$ for the full modular group $\mathrm{SL}_2(\ZZ)$.
Forms in $f\in \MM_k$ have an expansion
\begin{equation}\label{eq-modular_form_fourier_expansion}
    f(z) = \sum_{n\geq 0}a_f(n)   q^n
\end{equation}
with $a_f(n)\in \CC$, $q = e^{2\pi i z}$, $z\in \HH = \SET{z\;:\; \Im{z} > 0}$.
This expansion is referred to as the Fourier expansion of $f$ at the cusp.
One defines $v_\infty(f)$ as the smallest $n$ such that $a_f(n) \neq 0$.
Forms with $v_\infty(f) > 0$ form a subspace of $\MM_k$ called the space of cusp forms, and denoted $\SS_k$.
The space $\MM_k$ is the direct sum of $\SS_k$ and the one dimensional space spanned by the Eisenstein series of weight $k$:
$$
E_k(z) = 
\sum_{\gamma \in \Gamma_\infty \backslash \mathrm{SL}_2(\ZZ)}j(\gamma, z)^{-k}
=\frac{1}{2}\sum_{\substack{c,d\in \ZZ \\ \gcd(c,d) = 1}}\frac{1}{(cz + d)^k}.
$$

The Poincar\'e series $P_{k,m}$ is a modular form of weight $k$.
When $m=0$, we have that $P_{k,0} = E_k$ is the Eisenstein series of weight $k$. 
For $m \geq 1$, it is known that $P_{k,m}$ is a cusp form.

We define the Petersson inner product by $\BRAK{\cdot ,\cdot } : \MM_k\times \SS_k \rightarrow \CC$:
$$
\BRAK{f ,g }
= \int_{\mathrm{SL}_2(\ZZ) \backslash \HH}f(z)\overline{g(z)} \left(\Im z\right)^k \mathrm{d}\mu = \int_{\mathcal{F}}f(z) \overline{g(z)} \left(\Im z\right)^k \mathrm{d}\mu
$$
where as before $\mathrm{d}\mu$ is the hyperbolic measure $\mathrm{d}\mu = \frac{\mathrm{d}x \mathrm{d}y}{y^2}$, and $\mathcal{F}$ is the standard fundamental domain
\begin{multline*} 
\mathcal{F} = 
\SET{z\in \HH \; : \; |z| > 1, \Re{z} \in \left(-\frac{1}{2},\frac{1}{2}\right]} \\
\bigcup 
\SET{z\in \HH \; : \; |z| = 1, \arg{z} \in \left[\frac{\pi }{3}, \frac{\pi}{2}\right]}.
\end{multline*}
At times, we will also write $\BRAK{f,g}$ for more general functions, which we still define as $\int_{\mathcal{F}}f(z) \overline{g(z)} \left(\Im z\right)^k \mathrm{d}\mu$ .

An important property of the Poincar\'e series is their relationship with the Petersson inner product.
For a cusp form $f\in \SS_k$ with expansion $\sum_{n\geq 1}a_f(n) q^n$, and $m \geq 1$, we have
\begin{equation}\label{eq-inner_product_pkm}
    \BRAK{f ,P_{k,m} } = \frac{\Gamma(k-1)}{(4\pi m)^{k-1}}a_f(m).
\end{equation}

\subsection{Further notations}
We will use the strict $O$ notations, so that $f(x) = \BigO{g(x)}$ if one has $|f(x)| \leq C |g(x)|$ for some constant $C>0$ and for all valid inputs $x$.
We will also denote $f(x) \ll g(x)$ to indicate $f(x) = \BigO{g(x)}$ and $f(x)\asymp g(x)$ if both $f(x) = \BigO{g(x)}$ and $g(x) = \BigO{f(x)}$.
If the implied constants depend on some parameters $p_1,p_2,...$, this will either be explicitly stated, or indicated by subscripts such as $O_{p_1,_2,...}$ or $\ll_{p_1,p_2,...}$.

We write $f(x) = o(g(x))$ to indicate that for any $\epsilon>0$ there exists $x_0$ such that for all $x \geq x_0$ we have $|f(x)| \leq \epsilon |g(x)|$.
We will also denote $f(x) \lll g(x)$ to indicate $f(x) = o(g(x))$.
Lastly, we write $f(x) \sim g(x)$ to indicate $\lim_{x\rightarrow\infty} \frac{f(x)}{g(x)} = 1$.
Unless otherwise stated, the notations $o,\lll,\sim$ will refer to the variable $k$ (i.e. $k\rightarrow\infty$).

We will at times use $(a_1 ,...,a_n)$ to denote the gcd of $a_1,...,a_n$.
We will also use the notation $\sigma_s(n) = \sum_{d\mid n}d^s$ for the divisor functions.

Throughout, the variable $z$ will refer to a complex number of the form $x+iy$ with $x,y\in\RR$.

\section{Bounds for the Fourier coefficients of Poincar\'e series}\label{sec-pkm_coef}

\subsection{Main results of the section}

In this section we consider the Fourier coefficients of $P_{k,m}$ for $m\geq 1$ which we denote by 
$$
P_{k,m}(z) = \sum_{n\geq 1}\pkm{n}  q^n.
$$
The coefficients $\pkm{n} $ can be explicitly calculated, giving:
\begin{equation}\label{eq-pkm_fourier_coef}
\pkm{n} 
=   
\delta_{m,n} + 2\pi i^k \left(\frac{n}{m}\right)^{\frac{k-1}{2}}
\sum_{c > 0 }\frac{K(m,n,c)}{c}J_{k-1}\left(\frac{4\pi \sqrt{mn}}{c}\right)
\end{equation}
where $\delta_{m,n}$ is the Kronecker delta function, $J$ is the Bessel function of the first kind, and $K(m,n,c)$ is the Kloosterman sum:
\begin{equation}
K(m,n,c) = \sum_{\substack{1 \leq x \leq c \\ \gcd(x,c) = 1}}
e\left(\frac{mx + n \overline{x}}{c}\right)
\end{equation}
where $\overline{x}$ is the inverse of $x$ in $\left(\bigslant{\ZZ}{c\ZZ}\right)^*$.

The goal of this section is to provide bounds on $\left|\pkm{n}\right|$.
We denote
\begin{equation}\label{eq-smn}
S_{m,n} = \sum_{c > 0 }\frac{K(m,n,c)}{c}J_{k-1}\left(\frac{4\pi \sqrt{mn}}{c}\right).
\end{equation}
In \cite{MR0597120} Rankin computes an upper bound for $|S_{m,m}|$ which he then uses in order to prove a positive lower bound for
$$
\pkm{m} = 1 + 2\pi i^{k} S_{m,m}.
$$
We will slightly modify and improve his calculations in order to get upper bounds on $|S_{m,n}|$ for $n\neq m$, which in turn will give upper bounds on $\left|\pkm{n}\right|$.
Specifically, we will prove
\begin{proposition}\label{prop-pkmn_bound}
For $k\geq 16$, $m\geq 1$, $n\neq m$, and for any $\epsilon > 0$, we have:
\begin{equation*}
\pkm{n}
\ll_\epsilon
\left(\frac{n}{m} \right)^{\frac{k-1}{2}}
\left(\frac{(nm)^{\frac{1}{4} + \epsilon}}{k}
+ 
 (m,n)^{\frac{1}{2} }
\frac{k^{1/6}}{(nm)^{\frac{1}{4} - \epsilon}}\right).
\end{equation*}
\end{proposition}
In fact, we are going to prove a slightly stronger bound with the $\epsilon$ part replaced with more precise functions of $m,n,k$ (see \autoref{lem-smn_leq_Q}, \autoref{lem-smn_geq_Q}).
However we state this slightly weaker result since it will simplify later calculations.

We will also require the following result regarding $\pkm{m}$, which follows from Rankin's original computation in \cite{MR0597120}.
\begin{proposition}\label{prop-pkmm}
Let $\epsilon > 0$.
For $1\leq m \ll k^{2 - \epsilon}$ we have
$$
\pkm{m} \sim 1
$$
as $k\rightarrow\infty$.
\end{proposition}

\begin{remark}
\autoref{prop-pkmn_bound} can be significantly improved in the range $\sqrt{mn} \ll k$ (see \cite[Corollary 2.3]{MR1828743}).
However, we will not require this in this paper.
\end{remark}

\subsection{Preliminary lemmas and notations}
We begin by introducing some notations from \cite{MR0597120}.
Denote $\nu = k-1$.
Let $x_0$ the positive root of 
$$
x\cdot \exp{\frac{1}{3}(1-x^2)^{\frac{3}{2}}} = \frac{2}{e}
$$
so that $x_0\approx 0.629$.
Denote 
$$
x_\nu = \left(1 - \nu^{-\frac{2}{3}}\right)^{\frac{1}{2}},
\quad
y_\nu = \left(1 + \nu^{-\frac{2}{3}}\right)^{\frac{1}{2}}.
$$
We introduce the following functions
\begin{equation}
f(x) = 
\begin{cases}
    A_1 \nu^{-1/2}\left(\frac{1}{2}ex\right)^\nu & 0\leq x \leq x_0 \\
    A_2 \nu^{-1/2}\left(1- x^2\right)^{-1/4}\exp{-\frac{1}{3}\nu(1-x^2)^{3/2}} & x_0 \leq x \leq x_\nu \\
    \nu^{-1/3} & x_\nu \leq x \leq 1
\end{cases}
\end{equation}
where $A_1,A_2$ are constants chosen to make $f$ continuous (they do not depend on $\nu$).
And -
\begin{equation}
g(x) = 
\begin{cases}
    \nu^{-1/3} & 1 \leq x \leq y_\nu \\
    \nu^{-1/2}(x^2 - 1)^{-1/4} & y_\nu \leq x
\end{cases}.
\end{equation}

These functions are used to give upper bounds for the Bessel function.
\begin{lemma}\label{lem-bessel_J}
For $\nu \geq 15$ we have $J_\nu(\nu x) \ll f(x)$ when $0\leq x\leq 1$ and $J_\nu(\nu x) \ll g(x)$ when $x \geq 1$.
\end{lemma}
\begin{proof}
This is Lemma 4.4 from \cite{MR0597120}.
\end{proof}

We will also require a bound for Kloosterman sums.
\begin{lemma}\label{lem-k_mnc}
\begin{equation*}
\left|K(m,n,c)\right| \leq 2^{\omega_*(c/d)}c^{1/2}d^{1/2}
\end{equation*}
where $d = \gcd(m,n,c)$ and $\omega_*$ is slightly modified from the usual prime counting function and is defined by $\omega_*(2) = \frac{3}{2}$, $\omega_*(p)=1$ for other primes, and $\omega_*(n) = \sum_{p\mid n}\omega_*(p)$.
\end{lemma}
\begin{proof}
This is Lemma 3.1 from \cite{MR0597120}.
\end{proof}

\subsection{Bounds on \texorpdfstring{$S_{m,n}$}{Smn}}
We are now ready to give upper bounds for $S_{m,n}$. 
We use the notations from the previous section, and we further denote 
$$
Q = \frac{4\pi \sqrt{mn}}{\nu}.
$$
For a constant $B$, we will also denote by $\varepsilon_B$ the function 
$$
\varepsilon_B(x) = \exp{B\frac{\log (x+3)}{\log \log (2x+3)}}.
$$
We chose the notation $\varepsilon_B(x)$ since (for constant $B$) it grows slower than $x^{\epsilon}$ for all $\epsilon > 0$.

We are interested in giving bounds on $S_{m,n}$ (defined in \eqref{eq-smn}).
We begin by considering those summands in $S_{m,n}$ with $c\leq Q$.
\begin{lemma}\label{lem-smn_leq_Q}
For $k\geq 16$ we have
\begin{multline*}
\sum_{c \leq Q} \frac{K(m,n,c)}{c}
J_{k-1}\left(\frac{4\pi \sqrt{mn}}{c}\right)
\\ \ll
\varepsilon_{B_1}\left(Q\right)
\left(
\sigma_{-1/2}((m,n))\frac{(mn)^{1/4}}{k} + \sigma_{1/2}^Q((m,n))\frac{k^{1/6}}{(mn)^{1/4}}
\right)
\end{multline*}
for some absolute constant $ B_1$.
\end{lemma}
Here $\sigma_s$ is the divisor function $\sigma_s(n) = \sum_{d\mid n}d^s$, $(m,n)$ denotes the $\gcd$ of $m$ and $n$, and $\sigma_{s}^Q$ denotes
$$
\sigma_{s}^Q(n) = 
\sum_{\substack{d\mid n \\ d\leq Q}}d^{s}.
$$
\begin{proof}
Using \autoref{lem-k_mnc} we have
\begin{multline}\label{eq-bound1}
\left|
\sum_{c \leq Q} \frac{K(m,n,c)}{c}
J_{k-1}\left(\frac{4\pi \sqrt{mn}}{c}\right)
\right|  
 \\ \leq
\sum_{\substack{d\leq Q \\ d\mid (m,n)}}
\sum_{\substack{r \leq Q/d \\ \left(r, \frac{m}{d}, \frac{n}{d}\right)=1}}
2^{\omega_*(r)}r^{-1/2}\left|J_{\nu}\left(\nu \frac{Q}{rd}\right)\right|
 \\ \ll
Q^{-1/2}
\varepsilon_{B_1}(Q)
\sum_{\substack{d\leq Q \\ d\mid (m,n)}}
d^{1/2}
\sum_{r \leq Q/d}
\left(\frac{Q}{rd}\right)^{1/2}\left|J_{\nu}\left(\nu \frac{Q}{rd}\right)\right|
\end{multline}
where $B_1$ is chosen such that $2^{\omega_*(n)} \leq \varepsilon_{B_1}(x)$ for all  $n\leq x$ (the existence of such a constant follows from the Prime Number Theorem).
Denote
$$
T_d = \sum_{r \leq Q/d}
\left(\frac{Q}{rd}\right)^{1/2}\left|J_{\nu}\left(\nu \frac{Q}{rd}\right)\right|.
$$
From the computation following equation 5.10 in \cite{MR0597120} we have
$$
T_d \ll  \frac{Q}{d}\nu^{-1/2} + \nu^{-1/3}.
$$
Plugging this back in \eqref{eq-bound1} we get
\begin{align*}
&\sum_{c \leq Q} \frac{K(m,n,c)}{c} 
J_{k-1}\left(\frac{4\pi \sqrt{mn}}{c}\right)
 \\   
&\qquad \ll
Q^{-1/2} \varepsilon_{B_1}(Q)
\sum_{\substack{d\leq Q \\ d\mid (m,n)}}
\left(Q d^{-1/2} \nu^{-1/2} + d^{1/2} \nu^{-1/3}\right)
 \\ 
&\qquad =
\varepsilon_{B_1}(Q)\left(
Q^{1/2}\nu^{-1/2}\sigma_{-1/2}((m,n))
+
Q^{-1/2}\nu^{-1/3}\sigma_{1/2}^Q((m,n))\right)
 \\ 
&\qquad  \ll
\varepsilon_{B_1}(Q)
\left(\sigma_{-1/2}((m,n))\frac{(mn)^{1/4}}{k} + \sigma_{1/2}^Q((m,n))\frac{k^{1/6}}{(mn)^{1/4}}\right).
\end{align*}
\end{proof}

We now give two different bounds for the summands in \eqref{eq-smn} with $c \geq Q$.
\begin{lemma}\label{lem-smn_geq_Q}
For $k\geq 16$:
\begin{multline}\label{eq-lem_cgeqQ_1}
\sum_{c\geq Q} \frac{K(m,n,c)}{c}
J_{k-1}\left(\frac{4\pi \sqrt{mn}}{c}\right)
 \\ \ll
 \varepsilon_{B_2}(Q)
\left(\sigma_{-1/2}((m,n))\frac{(mn)^{1/4}}{k} + \sigma_{1/2}((m,n))\frac{k^{1/6}}{(mn)^{1/4}}\right)
\end{multline}
for some absolute constant $B_2$,
and
\begin{equation}\label{eq-lem_cgeqQ_2}
\sum_{c\geq Q} \frac{K(m,n,c)}{c}
J_{k-1}\left(\frac{4\pi \sqrt{mn}}{c}\right)
\ll \\
\left(\frac{Q}{k} + k^{-1/3}\right).
\end{equation}
\end{lemma}

\begin{proof}
\global\long\def\QQbound{(2Q+2)^2}
We have
\begin{equation}\label{eq-bound2}
\left|
\sum_{c \geq Q} \frac{K(m,n,c)}{c} 
J_{k-1}\left(\frac{4\pi \sqrt{mn}}{c}\right)
\right|  
\leq 
M + R
\end{equation}
with
\begin{align*}
&M = 
\sum_{Q \leq c\leq \QQbound}
\left|
\frac{K(m,n,c)}{c}
\right|
\left|
J_{k-1}\left(\frac{4\pi \sqrt{mn}}{c}\right) 
\right| \\
&R  = 
\sum_{c\geq \QQbound}
\left|
\frac{K(m,n,c)}{c}
\right|
\left|
J_{k-1}\left(\frac{4\pi \sqrt{mn}}{c}\right) 
\right|.
\end{align*}

We begin by considering $R$.
using the trivial bound $|K(m,n,c)| \leq c$ and \autoref{lem-bessel_J} we have
\begin{multline*}
R 
\leq 
\sum_{c\geq \QQbound}
\left|J_\nu\left(\nu\frac{Q}{c}\right)
\right|
\leq
\nu^{-1/2}
\left(\frac{1}{2}eQ\right)^\nu
\sum_{c\geq \QQbound}
\frac{1}{c^\nu}
 \\ \leq
\nu^{-1/2}(2Q+2)^{-\nu + 2}
\leq
k^{-1/2}\left(\frac{\sqrt{mn}}{k} + 2\right)^{-k + 2}
\end{multline*}
which is going to be negligible for $k>2$.

As for $M$, we give a bound in two different ways.
First, we can again use the trivial bound $|K(m,n,c)| \leq c$ and \autoref{lem-bessel_J} to get:
\begin{equation*}
M
\leq 
\sum_{Q \leq c\leq \QQbound}
\left|J_\nu\left(\nu\frac{Q}{c}\right)
\right|
\ll
\sum_{Q \leq c}
f\left(\frac{Q}{q}\right).
\end{equation*}
This last sum is estimated in \cite[eq. 5.12]{MR0597120}, which gives $M \ll\left( \frac{Q}{k} + k^{-1/3}\right)$, which proves \eqref{eq-lem_cgeqQ_2}. 

In order to prove \eqref{eq-lem_cgeqQ_1} we bound $M$ in a different way.
From \autoref{lem-k_mnc} we have 
$$
M\leq
\sum_{\substack{d\mid (m,n)}}
\sum_{\substack{ Q/d \leq r  \leq \QQbound}}
2^{\omega_*(r)}r^{-1/2}\left|J_{\nu}\left(\nu \frac{Q}{rd}\right)\right|
$$
This gives
\begin{equation}\label{eq-M_Td}
M \ll
Q^{-1/2}\varepsilon_{B_1}\left(\QQbound\right)
\sum_{d \mid (m,n)} d^{1/2} T_d 
\end{equation}
with $B_1$ chosen as in \autoref{lem-smn_leq_Q} and
$$
T_d = \sum_{r\geq Q/d}\left(\frac{Q}{rd}\right)^{1/2}
\left|
J_{\nu}\left(\nu \frac{Q}{rd}\right) 
\right|.
$$
From \autoref{lem-bessel_J} we have
\begin{align}\label{eq-td_bound}
\begin{split}
T_d \ll &
 \sum_{r\geq Q/d}  
\left(\frac{Q}{rd}\right)^{1/2}
f\left(\frac{Q}{rd}\right)
 \\
& \ll
\int_{Q/d}^\infty
\left(\frac{Q/d}{u}\right)^{1/2}
f\left(\frac{Q/d}{u}\right) \mathrm{d}u +  f(1)
\\
& =
 \left(\frac{Q}{d}\right)\int_0^1 x^{-\frac{3}{2}}f(x) \mathrm{d}x + \nu^{-1/3}
 \\
& = 
\nu^{-1/3} +  \left(\frac{Q}{d}\right)\nu^{-1/2}\int_0^{x_0}
x^{-\frac{3}{2}}\left(\frac{1}{2}ex\right)^\nu \mathrm{d}x
 \\
& 
\quad + \left(\frac{Q}{d}\right) \nu^{-1/2} \int_{x_0}^1
x^{-\frac{3}{2}} (1-x^2)^{-1/4}
\exp{\frac{1}{3}\nu(1-x^2)^{3/2}}\mathrm{d}x.
\end{split}
\end{align}

We have
$$
\int_0^{x_0}
x^{-\frac{3}{2}}\left(\frac{1}{2}ex\right)^\nu dx 
\ll
 \frac{\left(\frac{1}{2}e x_0\right)^\nu}{\nu}
<
 \frac{0.9^\nu}{\nu}
$$
since $\frac{1}{2}e x_0 < 0.9$.
Also, we have
\begin{multline*}
\int_{x_0}^1
x^{-\frac{3}{2}} (1-x^2)^{-1/4}
\exp{-\frac{1}{3}\nu(1-x^2)^{3/2}}\mathrm{d}x
 \\ \ll
 x_0^{-\frac{5}{2}}
\int_{x_0}^1
x (1-x^2)^{-1/4}
\exp{-\frac{1}{3}\nu(1-x^2)^{3/2}}\mathrm{d}x
\ll
\nu^{-1/2}
\end{multline*}
where the last inequality follows from the computation of $I_2$ on page 159 in \cite{MR0597120}.

Plugging these bounds back into \eqref{eq-td_bound}, we get
$$
T_d 
\ll  \nu^{-1}\frac{Q}{d} +  \nu^{-1/3}.
$$

And so, we have from \eqref{eq-M_Td} that
\begin{align*}
\begin{split}
&M \ll
 Q^{-1/2}\varepsilon_{B_1}\left(\QQbound\right)
\sum_{d \mid (m,n)} d^{1/2} T_d 
 \\ 
&\qquad \ll
\varepsilon_{B_1}\left(\QQbound\right)
\left(
\frac{Q^{1/2}}{\nu^{1/2}}\sigma_{-1/2}((m,n)) + 
\frac{Q^{-1/2}}{\nu^{1/3}}\sigma_{1/2}((m,n))\right)
 \\ 
&\qquad \ll
\varepsilon_{2B_1}(Q)
\left(\sigma_{-1/2}((m,n))\frac{(mn)^{1/4}}{k} + \sigma_{1/2}((m,n))\frac{k^{1/6}}{(mn)^{1/4}}\right).
\end{split}
\end{align*}

This completes the proof of \eqref{eq-lem_cgeqQ_1} with $B_2 = 2B_1$.
\end{proof}

We now show how \autoref{prop-pkmn_bound} and \autoref{prop-pkmm} follow from the estimates above.
\begin{proof}[Proof of \autoref{prop-pkmn_bound}]
For $n\neq m$ we have that 
$$
\pkm{n} = 2\pi i^k \left(\frac{n}{m}\right)^{\frac{k-1}{2}}S_{m,n}.
$$
We can bound $|S_{m,n}|$ using \autoref{lem-smn_leq_Q} and equation \eqref{eq-lem_cgeqQ_1} from \autoref{lem-smn_geq_Q}.
This gives
$$
\pkm{n}
\ll
 \varepsilon_B(Q)
\left(\sigma_{-1/2}((m,n))\frac{(mn)^{1/4}}{k} + \sigma_{1/2}((m,n))\frac{k^{1/6}}{(mn)^{1/4}}\right).
$$
The proposition then follows from
$$
\sigma_{-1/2}((m,n)) \ll_\epsilon (mn)^{\epsilon}
$$
and
$$
\sigma_{1/2}((m,n)) 
\ll (m,n)^{1/2}\sigma_{0}((m,n))
\ll_\epsilon(m,n)^{1/2 } (mn)^\epsilon.
$$
\end{proof}

\begin{proof}[Proof of \autoref{prop-pkmm}]
We have that
$$
\pkm{m} = 1 + 2\pi i^k S_{m,m}.
$$
We can bound $|S_{m,m}|$ using \autoref{lem-smn_leq_Q} and equation \eqref{eq-lem_cgeqQ_2} from \autoref{lem-smn_geq_Q}.
This gives
\begin{equation*}
S_{m,m}
\ll 
\varepsilon_B(Q)
\left(\sigma_{-1/2}(m)\frac{m^{1/2}}{k} + \sigma_{1/2}^Q(m)\frac{k^{1/6}}{m^{1/2}}\right)
+\frac{Q}{k} + k^{-1/3}.
\end{equation*}
We note also that in this case we have $Q = \frac{4\pi m}{k-1}$.

We use the bounds $\varepsilon_B(Q) \ll_\epsilon m^\epsilon$, $\sigma_{-1/2}(m) \ll_\epsilon m^\epsilon$, and
$$
\sigma_{1/2}^Q(m)
\leq Q^{1/2}\sigma_{0}(m)
\ll_\epsilon Q^{1/2}m^{\epsilon}.
$$
Using this we get
$$
S_{m,m}
\ll_\epsilon
\frac{m^{1/2 + \epsilon}}{k}
+ 
\frac{m^\epsilon}{k^{1/3}}
+
\frac{m}{k^2}
+ k^{-1/3}.
$$
So that $S_{m,m} \lll 1$ when $m \ll k^{2 - 3\epsilon}$ and the result follows.
\end{proof}

\section{Mass equidistribution of \texorpdfstring{$P_{k,m}$}{Pkm}}\label{sec-mass}
\subsection{The unfolding method}
We now turn our attention back to the mass equidistribution of $P_{k,m}$.
For this we follow the standard method of considering the inner product of $\left|P_{k,m}(z)\right|^2$ against function in the spectral decomposition of $\Delta$ on $L^2\left(\mathrm{SL}_2(\mathbb{Z}) \backslash \mathbb{H}\right)$.
The spectrum is known to consist of three types of functions:
\begin{enumerate}
    \item The constant function $\phi(z) = \sqrt{\frac{3}{\pi}}$.
    \item A sequence $\phi_i(z)$ of Maass cusp forms with associated eigenfunctions $\lambda_i$ so that $\Delta \phi_i +\lambda_i \phi_i = 0$.
    \item The unitary Eisenstein series $\phi_t(z) = E\left(z, \frac{1}{2} + it\right)$ for $t\in \RR$ which satisfy $\Delta \phi_t + \left( \frac{1}{4} + t^2\right)\phi_t = 0$.
\end{enumerate}
The inner product with the constant function corresponds to normalizing our functions $P_{k,m}$.
We then wish to show that the inner product of the normalized functions against the non-constant function in the spectrum tends to $0$ with $k$.
In other words, we wish to show that
$$
\frac{\BRAK{P_{k,m}\phi,P_{k,m}}}{\BRAK{P_{k,m},P_{k,m}}}
\xrightarrow{k\rightarrow\infty} 0
$$
where $\phi$ is any non-constant function in the spectrum.

In order to estimate $\BRAK{P_{k,m},P_{k,m}}$ we use property \eqref{eq-inner_product_pkm} regarding the way the Poincar\'e series behave with the Petersson inner product.
Together with \autoref{prop-pkmm} we conclude that for $m \ll k^{2 - \epsilon}$:
\begin{equation}\label{eq-pkm_normalization}
\BRAK{P_{k,m},P_{k,m}} = \frac{\Gamma\left(k-1\right)}{\left(4\pi m \right)^{k-1}}\pkm{m}\sim \frac{\Gamma\left(k-1\right)}{\left(4\pi m \right)^{k-1}}.
\end{equation}

We now wish to estimate $\BRAK{P_{k,m}\phi,P_{k,m}}$.
For this we use the unfolding method:
\begin{align}\label{eq-unfolding}
\begin{split}
\BRAK{P_{k,m}\phi,P_{k,m}} & =
\int_{\mathcal{F}}
P_{k,m}(z) \phi(z) 
\overline{P_{k,m}(z)} y^k
\mathrm{d}\mu
 \\ & =
\int_{\mathcal{F}}
P_{k,m}(z) \phi(z) y^k
\sum_{\gamma\in\Gamma_\infty \backslash \SL}
\frac{\overline{e(m\gamma z)}}{\overline{j(\gamma,z)^k}}
\mathrm{d}\mu
\\ & =
\int_{\mathcal{F}}
\sum_{\gamma\in\Gamma_\infty \backslash \SL}
\phi(\gamma z)
\frac{P_{k,m}(\gamma z)}{j(\gamma,z)^k}
\frac{\overline{e(m\gamma z)}}{\overline{j(\gamma,z)^k}}
y^k \mathrm{d}\mu 
\\ & =
\int_{\mathcal{F}}
\sum_{\gamma\in\Gamma_\infty \backslash \SL}
\phi(\gamma z) P_{k,m}(\gamma z)
\overline{e(m\gamma z)} \Im(\gamma z)^k \mathrm{d}\mu
\\ & =
\int_0^\infty\int_0^1
\phi(z)P_{k,m}(z)
e^{-2\pi i m x}
e^{-2\pi m y} y^k  \frac{\mathrm{d}x \, \mathrm{d}y }{y^2}.
\end{split}
\end{align}

For a smooth function $f:\HH\rightarrow\CC$ satisfying $f(z + 1) = f(z)$, we denote its $n$-th Fourier coefficient at height $y$ by $\widehat{f}[n](y)$, so that 
$$
f(x+ iy) = \sum_{n\in \ZZ}\widehat{f}[n](y)e(nx).
$$

With this notation, \eqref{eq-unfolding} becomes
\begin{equation}\label{eq-int_m_fourier}
\BRAK{P_{k,m}\phi,P_{k,m}}
=
\int_0^\infty
\widehat{\phi \cdot P_{k,m}}[m](y) e^{-2\pi m y}y^{k-2} dy.  
\end{equation}
In the following subsections we will estimate this integral, both in the case where $\phi$ is a Maass cusp form (\autoref{prop-inner_prod_maass}) and in the case where $\phi$ is a unitary Eisenstein series (\autoref{prop-inner_prod_eisenstein}).
We will then prove \autoref{thm-mass_equi} using these estimates.

\subsection{Inner product with Hecke-Maass cusp forms}
We first consider \eqref{eq-int_m_fourier} in the case where $\phi$ is a Maass cusp form.
Our main result is the following proposition:
\begin{proposition}\label{prop-inner_prod_maass}
Let $\phi$ be a Maass cusp form, and let $\epsilon > 0$.
Then for any $m(k)$ satisfying $1 \leq  m \ll k^{2-\epsilon}$ we have
$$
\frac{\BRAK{P_{k,m}\phi, P_{k,m}}}
{\BRAK{P_{k,m}, P_{k,m}}}
\ll_{\phi,\epsilon}
\left(
\frac{m^{1 +  \epsilon}}{k^{\frac{3}{2}}}
+
\frac{m^{ \epsilon}}{ k^{\frac{1}{3}}}
\right).
$$
\end{proposition}

\begin{proof}
Let $\epsilon >0$, let $m(k)$ be a function satisfying $1 \leq m \ll k^{2-\epsilon}$, and let $\phi$ be a normalized Maass cusp form.
We can assume that $\phi$ is a Hecke Maass cusp form, by choosing an appropriate basis for each eigenspace.

We first consider the case where $\phi$ is an odd Maass cusp forms, i.e. $\phi(-\overline{z}) = -\phi(z)$.
Since $P_{k,m}(-\overline{z}) = \overline{P_{k,m}(z)}$ we can deduce that 
$$
\left|P_{k,m}(-\overline{z})\right|^2 = \left|P_{k,m}(z)\right|^2.
$$
From the definition of the Petersson inner product it then follows that
$$
\BRAK{P_{k,m}\phi, P_{k,m}}
=
\BRAK{\phi, \left|P_{k,m}\right|^2}
=0.
$$

Therefore we are left with the case where $\phi$ is an even Hecke Maass cusp form, i.e. $\phi(-\overline{z}) = \phi(z)$.
In this case $\phi$ has the form
$$
\phi(z) = \rho y^{1/2}\sum_{n\neq 0}\lambda(|n|)K_{it}(2\pi|n|y)e( n x).
$$
Here $\frac{1}{4} + t^2$ is the eigenvalue of $\phi$ with respect to $-\Delta$, $K$ is the modified Bessel function of the third kind, $\rho>0$ is a normalization factor, and $\lambda(|n|)$ are the eigenvalues of $\phi$ with respect to the Hecke operators.

We will ignore the normalization factor $\rho$, since we will regard $\phi$ as being fixed.
We will also require the Rankin-Selberg second moment bound \cite{MR411},\cite{MR2626}:
\begin{equation}\label{eq-rs_second_moment_full}
\sum_{0 < n\leq x} \lambda^2(n) \ll_\phi x
\end{equation}
which by partial summation implies
\begin{equation}\label{eq-rs_second_moment}
\sum_{0 < n\leq x} \frac{\lambda^2(n)}{n} \ll_{\phi,\epsilon} x^{\epsilon} 
\end{equation}
and
\begin{equation}\label{eq-rs_second_moment2}
\sum_{n=1}^{\infty} \frac{\lambda^2(n)}{n^{1+\epsilon}} \ll_{\phi,\epsilon} 1. 
\end{equation}
We now wish to estimate $\widehat{\phi P_{k,m}}[m]$.
By the product convolution property of the Fourier transform, we have that
$$
\widehat{\phi P_{k,m}}[m](y) =
\sum_{\substack{r + s = m \\ r,s \in \ZZ}}
\widehat{P_{k,m}}[r](y)\cdot \widehat{\phi}[s](y)
=
a_1(y) +   a_2(y)
$$ 
with
\begin{align*}
a_1(y)
=&
\;y^{\frac{1}{2}}\sum_{\substack{r + s = m \\ r,s \geq 1}}
\pkm{r}e^{-2\pi r y} \lambda(s) K_{it}(2\pi s y), \\
a_2(y) 
=&
\;y^{\frac{1}{2}}\sum_{\substack{r - s = m \\  r\geq m + 1 \\ s \geq 1}}
\pkm{r}e^{-2\pi r y} \lambda(s) K_{it}(2\pi s y)
\end{align*}
where we used that for $r\geq 1, s \neq 0$:
$$
\widehat{P_{k,m}}[r](y) = \pkm{r}e^{-2\pi r y},\quad
\widehat{\phi}[s](y) = y^{1/2}\lambda(|s|) K_{it}(2\pi |s| y).
$$
Plugging this into \eqref{eq-int_m_fourier} we get
\begin{equation}\label{eq-brak_ineq_a12}
\left|\BRAK{P_{k,m}\phi, P_{k,m}}\right|
\leq
\left|
\int_0^\infty
a_1(y) e^{-2\pi m y}y^{k-1} \frac{\mathrm{d}y}{y}
\right|
+
\left|
\int_0^\infty
a_2(y) e^{-2\pi m y}y^{k-1} \frac{\mathrm{d}y}{y}
\right|.
\end{equation}

We now estimate each of these integrals individually.
We begin with the integral involving $a_1$.
We have the following uniform bound for the $K$ Bessel function of imaginary order \cite[Corollary 3.2]{MR3102912}:
\begin{equation}\label{eq-K_bessel_bound}
K_{it}(x) \ll \frac{e^{-x}}{\sqrt{x}}.    
\end{equation}

Using this bound we get
\begin{multline*} 
a_1(y) 
\ll
\sum_{n=1}^{m-1}
\left|\pkm{n}\right|
\left| \frac{\lambda(m-n)}{(m-n)^{1/2}}\right| 
e^{-2\pi n y}\cdot e^{-2\pi (m-n)y} \\
\ll_{\epsilon}
e^{-2\pi m y}(M_1 + E_1)
\end{multline*}
where
\begin{align*}
& M_1= 
\sum_{n=1}^{m-1}
\left(\frac{n}{m}\right)^{\frac{k-1}{2}}
\frac{(nm)^{\frac{1}{4} + \epsilon}}{k}
\left| \frac{\lambda(m-n)}{(m-n)^{1/2}}\right| \\
& E_1 = 
\sum_{n=1}^{m-1}
\left(\frac{n}{m}\right)^{\frac{k-1}{2}}
\frac{k^{1/6}(m,n)^{\frac{1}{2}}}{(nm)^{\frac{1}{4} - \epsilon}}
\left| \frac{\lambda(m-n)}{(m-n)^{1/2}}\right|.
\end{align*}

We first consider $M_1$.
Using the Cauchy–Schwarz inequality we get
$$
M_1 \ll \frac{m^{\frac{1}{2} + 2\epsilon}}{k} 
\left(\sum_{n=1}^{m-1}\left(\frac{n}{m}\right)^{k - \frac{1}{2} + 2\epsilon}\right)^{\frac{1}{2}}
\cdot
\left(\sum_{n=1}^{m-1}\frac{\lambda^2(n)}{n}\right)^{\frac{1}{2}}.
$$
We use the Rankin-Selberg second moment estimate \eqref{eq-rs_second_moment} to bound the second sum
$$
\sum_{n=1}^{m-1}\frac{\lambda^2(n)}{n} \ll_{\phi, \epsilon} m^{\epsilon}.
$$
We bound the first sum using an integral.
Specifically, if $f$ is a smooth non-negative function such that $f'$ changes signs a finite amount of times (say $C$ times), then
\begin{equation}\label{eq-int_bound}
\sum_{a}^{b}f(n) \leq \int_{a}^{b}f(w)\mathrm{d}w + 2C\mathop{\max}_{n\in[a,b]}f(n).
\end{equation}
Using this we find that
\begin{equation*}
\sum_{n=1}^{m-1}\left(\frac{n}{m}\right)^{k - \frac{1}{2} + 2\epsilon}
\ll
m\int_{0}^{1}w^{k - \frac{1}{2} + 2\epsilon} \mathrm{d}w \; + \; 1
\ll \frac{m}{k} + 1.
\end{equation*}
Combining these estimates we get:
\begin{equation*}
M_1  \ll
\frac{m^{\frac{1}{2} + 2\epsilon}}{k} 
\cdot
\left( \frac{m}{k} + 1\right)^{\frac{1}{2}}
\cdot 
m^{\epsilon}
\ll_\epsilon \frac{m^{1 + 3\epsilon}}{k^{\frac{3}{2}}}
+ \frac{m^{\frac{1}{2} + 3\epsilon}}{k}.
\end{equation*}

Bounding $E_1$ is done in a similar fashion after enumerating over the value of $(m,n)$:
\begin{align*}
E_1 
& \ll 
\frac{k^{1/6}}{m^{\frac{1}{2} - 2\epsilon}}
\sum_{n=1}^{m-1}
(n,m)^{\frac{1}{2} }
\left(\frac{n}{m}\right)^{\frac{k}{2} - \frac{3}{4} + \epsilon}
\left| \frac{\lambda(m-n)}{(m-n)^{1/2}}\right|
\\& \leq
\frac{k^{1/6}}{m^{\frac{1}{2} - 2\epsilon}}
\left(\sum_{n=1}^{m-1}\frac{\lambda^2(n)}{n}\right)^{\frac{1}{2}}
\left(\sum_{n=1}^{m-1}\left(\frac{n}{m}\right)^{k - \frac{3}{2} + 2\epsilon}(n,m)\right)^{\frac{1}{2}}
\\& \ll_\phi
\frac{k^{1/6}}{m^{\frac{1}{2} - 2\epsilon}}
\cdot m^{\epsilon} \cdot 
\left(\sum_{g\mid m} g  
\sum_{n' = 1}^{m/g - 1} \left(\frac{n'}{m/g}\right)^{k - \frac{3}{2} + 2\epsilon}
\right)^{\frac{1}{2}}
\\& \ll_\epsilon
\frac{k^{1/6}}{m^{\frac{1}{2} - 3\epsilon}}
\left(
\sum_{g\mid m} g \left(\frac{m/g}{k} + \left(1 - \frac{g}{m}\right)^{k - \frac{3}{2} + 2\epsilon}\right)\right)^{\frac{1}{2}}
\\&  \ll_\epsilon
\frac{k^{1/6}}{m^{\frac{1}{2} - 3\epsilon}}
\left(\sigma_0(m) \frac{m}{k} \right)^{\frac{1}{2}}
\ll_\epsilon
\frac{m^{4\epsilon}}{k^{\frac{1}{3}}}
\end{align*}
where we used the bound $g\left(1 - \frac{g}{m}\right)^k \ll \frac{m}{k}$.

It follows that 
$$
a_1(y) \ll_{\phi,\epsilon}
e^{-2\pi m y}
\left( \frac{m^{1 + 3\epsilon}}{k^{\frac{3}{2}}}
+ 
\frac{m^{\frac{1}{2} + 3\epsilon}}{k}
+
\frac{m^{ 4\epsilon}}{k^{\frac{1}{3}}}
\right)
\ll
e^{-2\pi m y}
\left( \frac{m^{1 + 3\epsilon}}{k^{\frac{3}{2}}}
+
\frac{m^{ 4\epsilon}}{k^{\frac{1}{3}}}
\right).
$$
Thus, we have
\begin{align}\label{eq-int_a_1}
\begin{split}
&\int_0^\infty
a_1(y)e^{-2\pi m y}y^{k-1} \frac{\mathrm{d}y}{y}
  \ll_{\phi,\epsilon}
\left( \frac{m^{1 + 3\epsilon}}{k^{\frac{3}{2}}}
+ 
\frac{m^{ 4\epsilon}}{k^{\frac{1}{3}}}
\right)
\int_0^\infty
e^{-4\pi m y} y^{k - 1}\frac{\mathrm{d}y}{y} 
\\ & \qquad 
= \left( \frac{m^{1 + 3\epsilon}}{k^{\frac{3}{2}}}
+
\frac{m^{ 4\epsilon}}{k^{\frac{1}{3}}}
\right)
\frac{\Gamma\left(k - 1\right)}{(4\pi m)^{k - 1}}.
\end{split}
\end{align}

We now consider $a_2$.
Once more using the bound \eqref{eq-K_bessel_bound} for $K_{it}(x)$ we get
\begin{multline*}
a_2(y)  \\
\ll_{\phi,\epsilon}
\sum_{n > m }
\left( \frac{n}{m} \right)^{\frac{k-1}{2}}
\left( 
\frac{(nm)^{\frac{1}{4} + \epsilon}}{k}
+ \frac{k^{1/6}(n,m)^{\frac{1}{2}}}{(nm)^{\frac{1}{4} - \epsilon}}
\right)
\left| \frac{\lambda(m-n)}{(m-n)^{1/2}}\right|
e^{-4\pi n y}e^{2\pi m y}
 \\ \ll
e^{2\pi m y}\left(M_2(y) + E_2(y)\right)
\end{multline*}
with 
\begin{align*}
& M_2(y) = 
\frac{m^{\frac{1}{2} + 2\epsilon}}{k}
\sum_{n > m}
\left(\frac{n}{m}\right)^{\frac{k}{2} - \frac{1}{4} + \epsilon}
\left| \frac{\lambda(n-m)}{(n-m)^{1/2}}\right|
\exp{-4\pi n y }
\\
& E_2(y) = 
\frac{k^{1/6}}{m^{\frac{1}{2} - 2\epsilon}}
\sum_{n > m} (n,m)^{\frac{1}{2}}
\left(\frac{n}{m}\right)^{\frac{k}{2} - \frac{3}{4} + \epsilon}
\left| \frac{\lambda(n-m)}{(n-m)^{1/2}}\right|
\exp{-4\pi n y}
\end{align*}

We remind the reader that we are interested in bounding
\begin{align*}
&
\int_0^\infty a_2(y)e^{-2\pi m y}y^{k-2} \mathrm{d}y 
 \\ & \qquad \ll_{\epsilon}
\int_0^\infty e^{2\pi m y}\left(M_2(y) + E_2(y)\right)e^{-2\pi m y}y^{k-2} \mathrm{d}y 
 \\ & \qquad =
\int_0^\infty M_2(y)y^{k-2} \mathrm{d}y 
+
\int_0^\infty E_2(y)y^{k-2} \mathrm{d}y .
\end{align*}

We first consider the integral involving $M_2$.
\begin{align*}
&\int_0^\infty M_2(y)y^{k-2} \mathrm{d}y 
\\ 
&\qquad =
\frac{m^{\frac{1}{2} + 2\epsilon}}{k}
\sum_{n > m}
\left(\frac{n}{m}\right)^{\frac{k}{2} - \frac{1}{4} + \epsilon}
\left| \frac{\lambda(n-m)}{(n-m)^{1/2}}\right|
\int_{0}^{\infty}
\exp{-4\pi m y \frac{n}{m}}y^{k-2} \mathrm{d}y
\\
&\qquad =
\frac{\Gamma(k-1)}{(4\pi m)^{k-1}}
\frac{m^{\frac{1}{2} + 2\epsilon}}{k}
\sum_{n > m}
\left(\frac{n}{m}\right)^{\frac{k}{2} - \frac{1}{4} + \epsilon}
\left| \frac{\lambda(n-m)}{(n-m)^{1/2}}\right|
\left(\frac{m}{n}\right)^{k-1}.
\end{align*}
Using Cauchy-Schwarz we then get
\begin{align*}
\int_0^\infty M_2(y)y^{k-2} \mathrm{d}y & \ll \frac{\Gamma(k-1)}{(4\pi m)^{k-1}} \frac{m^{\frac{1}{2} + 3\epsilon}}{k}
\\ 
&\times 
\left(\sum_{n>m} \left(\left(\frac{n}{m}\right)^{-\frac{k}{2}+\frac{3}{4}+\epsilon} \left(\frac{n}{m} - 1\right)^{\epsilon}\right)^2
\right)^{\frac{1}{2}}
\left(
\sum_{n > m}
\frac{\lambda^2(n)}{n^{1 + 2\epsilon}}
\right)^{\frac{1}{2}}.
\end{align*}
From \eqref{eq-rs_second_moment2} we have that 
$$
\sum_{n > m}
\frac{\lambda^2(n)}{n^{1 + 2\epsilon}}
\ll_\epsilon 1.
$$
For the first sum we use \eqref{eq-int_bound} to bound the sum using an integral.
We get that
\begin{align*}
&\sum_{n>m} \left(\frac{n}{m}\right)^{-k+\frac{3}{2}+2\epsilon} \left(\frac{n}{m} - 1\right)^{2\epsilon}
\\
& \qquad \ll
m\int_0^\infty w^{2\epsilon}(w+1)^{-k+\frac{3}{2}+2\epsilon} \mathrm{d}w
+ \left(\frac{1}{m}\right)^{2\epsilon}
\\
& \qquad \ll
m \cdot B\left( k - \frac{5}{2} - 4\epsilon, 1 + 2\epsilon\right) + 1
\ll \frac{m}{k} + 1.
\end{align*}
where $B(\cdot,\cdot)$ is the beta function.
We used the bound $B(k,C) \ll_C k^{-C}$ which follows from Stirling's approximation.

Plugging this approximation back into the integral involving $M_2$ we get
$$
\int_0^\infty M_2(y)y^{k-2} \mathrm{d}y
\ll_\epsilon
\frac{\Gamma(k-1)}{(4\pi m)^{k-1}} 
\left(
\frac{m^{1 + 3\epsilon}}{k^{\frac{3}{2}}}
+
\frac{m^{\frac{1}{2} + 3\epsilon}}{k}
\right).
$$

We now evaluate the integral involving $E_2$.
We have
\begin{align*}
&\int_0^\infty E_2(y)y^{k-2} \mathrm{d}y 
\\ 
&\qquad =
\frac{k^{\frac{1}{6}}}{m^{\frac{1}{2}-2\epsilon}}
\sum_{n > m} (n,m)^{\frac{1}{2}}
\left(\frac{n}{m}\right)^{\frac{k}{2} - \frac{3}{4} + \epsilon}
\left| \frac{\lambda(n-m)}{(n-m)^{1/2}}\right|
\int_{0}^{\infty}
e^{-4\pi n y}y^{k-2} \mathrm{d}y
\\
&\qquad =
\frac{\Gamma(k-1)}{(4\pi m)^{k-1}}
\frac{k^{\frac{1}{6}}}{m^{\frac{1}{2}-2\epsilon}}
\sum_{n > m}
\left(\frac{n}{m}\right)^{\frac{k}{2} - \frac{1}{4} + \epsilon}
\left| \frac{\lambda(n-m)}{(n-m)^{1/2}}\right|
\left(\frac{m}{n}\right)^{k-1}.
\end{align*}
And so, a similar computation to the case of $M_2$ gives:
\begin{align*}
&
\left(\frac{\Gamma(k-1)}{(4\pi m)^{k-1}}
\frac{k^{\frac{1}{6}}}{m^{\frac{1}{2}-3\epsilon}}\right)^{-1}
\int_0^\infty E_2(y)y^{k-2} \mathrm{d}y =
\\ 
&\qquad =
\left(\sum_{n>m}\frac{\lambda^2(n)}{n^{1+2\epsilon}}\right)^{\frac{1}{2}}
\left(\sum_{g\mid m}g
\sum_{n'> m/g } 
\left(\frac{n'}{m/g}\right)^{-k + \frac{3}{2} + 2\epsilon}
\left(\frac{n'}{m/g} - 1\right)^{2\epsilon}
\right)^{\frac{1}{2}}
\\ & \qquad \ll_{\phi,\epsilon}
\left(
\sum_{g\mid m} m\int_0^\infty w^{2\epsilon}(w+1)^{-k+\frac{3}{2}+2\epsilon} \mathrm{d}w
+ g\left(1+\frac{g}{m}\right)^{-k + \frac{3}{2} + 2\epsilon}
\right)^{\frac{1}{2}}
\\ & \qquad \ll_\epsilon
\left(
\sigma_0(m) m B\left( k - \frac{5}{2} - 4\epsilon, 1 + 2\epsilon\right)
+\sigma_0(m) \frac{m}{k}
\right)^{\frac{1}{2}}
\ll_\epsilon
\frac{m^{\frac{1}{2} + \epsilon}}{k^{\frac{1}{2}}}.
\end{align*}

And so, combining these results we conclude
\begin{equation*}
\int_0^\infty a_2(y)e^{-2\pi m y}y^{k-2} \mathrm{d}y  
 \ll_{\phi,\epsilon}
\frac{\Gamma\left(k - 1\right)}{(4\pi m)^{k -1}}
\left(
\frac{m^{1 +  3\epsilon}}{k^{\frac{3}{2}}}
+
\frac{m^{ 4\epsilon}}{ k^{\frac{1}{3}}}
\right).
\end{equation*}
This is the same estimate we got in \eqref{eq-int_a_1} for $\int_0^\infty a_1(y)e^{-2\pi m y}y^{k-\frac{3}{2}} dy$.
And so, using \eqref{eq-brak_ineq_a12} we get
\begin{equation}\label{eq-inner_pkm_maass}
\BRAK{P_{k,m}\phi, P_{k,m}}
\\   \ll_{\phi,\epsilon} 
\frac{\Gamma\left(k - 1\right)}{(4\pi m)^{k -1}}
\left(
\frac{m^{1 +  3\epsilon}}{k^{\frac{3}{2}}}
+
\frac{m^{ 4\epsilon}}{ k^{\frac{1}{3}}}
\right).
\end{equation}
Using the estimation \eqref{eq-pkm_normalization} for $\BRAK{P_{k,m}, P_{k,m}}$, we conclude
$$
\frac{\BRAK{P_{k,m}\phi, P_{k,m}}}
{\BRAK{P_{k,m}, P_{k,m}}}
\ll_{\phi,\epsilon}
\left(
\frac{m^{1 +  \epsilon}}{k^{\frac{3}{2}}}
+
\frac{m^{ \epsilon}}{ k^{\frac{1}{3}}}
\right).
$$
as required.
\end{proof}

\subsection{Inner product with unitary Eisenstein series}
We now consider \eqref{eq-int_m_fourier} in the case where $\phi$ is a unitary Eisenstein series.
Our main result is the following proposition:
\begin{proposition}\label{prop-inner_prod_eisenstein}
Let $\phi(z) = E\left(z, \frac{1}{2} + it\right)$ be a unitary Eisenstein series, 
and let $\epsilon > 0$.
Then for any $m(k)$ satisfying $1 \leq m \ll k^{2 - \epsilon}$ we have
$$
\frac{\BRAK{P_{k,m}\phi, P_{k,m}}}
{\BRAK{P_{k,m}, P_{k,m}}}
\ll_\epsilon
(1 +|t|)^{\epsilon}
\left( 
\frac{k^{\frac{1}{2}}}{m^{\frac{1}{2}}} +
\frac{m^{1 + \epsilon}}{k^{\frac{3}{2}}} + 
\frac{m^{\epsilon}}{k^{\frac{1}{3}}}
\right).
$$
\end{proposition}
\begin{proof}
Let $\epsilon >0$, let $m(k)$ be a function satisfying $1 \leq m \ll k^{2 - \epsilon}$, and let $\phi(z)= E\left(z, \frac{1}{2} + it\right)$ be a unitary Eisenstein series.
The Fourier coefficients of $\phi(z)$ are known, so that we have
$$
\phi(z) =
y^{\frac{1}{2} + it} +\frac{\xi(it)}{\xi(1 + it)} y^{\frac{1}{2} - it} 
+ 
\frac{ y^{\frac{1}{2}}}{\xi(1 + 2it)}   \sum_{n\neq 0}\lambda(n)K_{it}(2\pi |n| y) e(nx).
$$
where
\begin{equation*}
\lambda(n) = |n|^{it}\sigma_{-2it}(|n|),\quad
\xi(s) = \pi^{-s/2}\Gamma(s/2)\zeta(s).
\end{equation*}

We write
$$
\phi(z) = \widehat{\phi}[0](y) + \varphi(z)
$$
where
\begin{align*}
& \widehat{\phi}[0](y) = y^{\frac{1}{2} + it} +\frac{\xi(it)}{\xi(1 + it)} y^{\frac{1}{2} - it}  \\
& \varphi(z) = \sum_{n\neq 0 } \widehat{\phi}[n](y)e(nx) 
= \frac{ y^{\frac{1}{2}}}{\xi(1 + 2it)}   \sum_{n\neq 0}\lambda(n)K_{it}(2\pi |n| y) e(nx).  
\end{align*}
Once more, from \eqref{eq-int_m_fourier}, we are interested in estimating 
$$
\int_0^\infty
\widehat{\phi \cdot P_{k,m}}[m](y) e^{-2\pi m y}y^{k-2} \mathrm{d}y.  
$$
We can separate this Fourier coefficient into two parts:
$$
\widehat{\phi \cdot P_{k,m}}[m](y)
=
\left(
\widehat{\phi}[0](y)\cdot \widehat{P_{k,m}}[m](y)
\right)
+
\left(
\widehat{\varphi \cdot P_{k,m}}[m](y)
\right).
$$
From the definition of $\lambda(n)$ we have 
$$
\lambda(n) \leq \sigma_0(|n|) \ll_\epsilon |n|^\epsilon
$$
for any $\epsilon>0$, so that the bounds \eqref{eq-rs_second_moment}, \eqref{eq-rs_second_moment2} are satisfied in this case as well.
And so, giving a bound for
\begin{equation}\label{eq-int_pkm_varphi}
\int_0^\infty\widehat{\varphi \cdot P_{k,m}}[m](y) e^{-2\pi m y}y^{k-2} \mathrm{d}y    
\end{equation}
can be accomplished using the exact same computations from \autoref{prop-inner_prod_maass} (where instead of $\varphi$ we had an even Hecke Maass cusp form).
Furthermore, since the bound $\lambda(n) \ll_\epsilon |n|^\epsilon$ is uniform in $t$, the calculations from the previous section will also be uniform in $t$. 
And so, from the computations leading to \eqref{eq-inner_pkm_maass} we get:
\begin{multline*}
\frac{1}{\BRAK{P_{k,m},P_{k,m}}}
\int_0^\infty
\widehat{\varphi \cdot P_{k,m}}[m](y) e^{-2\pi m y}y^{k-2} \mathrm{d}y 
\\ \ll_{\epsilon} 
\frac{1}{\xi(1 + it)}
\left(
\frac{m^{1 + 3\epsilon}}{k^{\frac{3}{2}}}
+
\frac{m^{  4\epsilon}}{ k^{\frac{1}{3} }}
\right).
\end{multline*}
Using the bound $|\zeta(1 + it)|^{-1} \ll \log(|t| + 1)$ and Stirling's approximation then gives
\begin{multline}\label{eq-first_part_eisenstein_approx}
\frac{1}{\BRAK{P_{k,m},P_{k,m}}}
\int_0^\infty
\widehat{\varphi \cdot P_{k,m}}[m](y) e^{-2\pi m y}y^{k-2} \mathrm{d}y 
\\ \ll_{\epsilon} 
\log(|t| + 1)
\left(
\frac{m^{1 + 3\epsilon}}{k^{\frac{3}{2}}}
+
\frac{m^{  4\epsilon}}{ k^{\frac{1}{3} }}
\right).
\end{multline}

We now consider 
\begin{equation*}
\int_0^\infty
\left(
\widehat{\phi}[0](y)\cdot \widehat{P_{k,m}}[m](y)
\right) e^{-2\pi m y}y^{k-2} \mathrm{d}y .   
\end{equation*}
We have
$$
\widehat{\phi}[0](y) = y^{\frac{1}{2} + it} +\frac{\xi(it)}{\xi(1 + it)} y^{\frac{1}{2} - it}.
$$
Using the functional equation $\xi(s) = \xi(1-s)$ together with Stirling's approximation and the bounds 
$$
\log(|t| + 1)^{-1} \ll |\zeta(1 + it)| \ll \log(|t| + 1)
$$
we get $\frac{\xi(it)}{\xi(1 + it)} \ll \log^2(|t| + 1)$, so that 
$$
\widehat{\phi}[0](y) \ll (1 +  |t|)^\epsilon y^{\frac{1}{2}}.
$$

It follows that
\begin{multline*}
\int_0^\infty
\left(
\widehat{\phi}[0](y)\cdot \widehat{P_{k,m}}[m](y)
\right) e^{-2\pi m y}y^{k-2} \mathrm{d}y 
\\
\ll_\epsilon
 (1 +  |t|)^\epsilon
\int_0^\infty\left( y^{\frac{1}{2}} e^{-2\pi m y}\right) e^{-2\pi m y}y^{k-2} \mathrm{d}y 
=
(1 +  |t|)^\epsilon
\frac{\Gamma\left(k - \frac{1}{2}\right)}{(4\pi m)^{ k - \frac{1}{2}}}.
\end{multline*}
Using the estimation \eqref{eq-pkm_normalization} for $\BRAK{P_{k,m}, P_{k,m}}$, and the approximation
$$
\bigslant{\frac{\Gamma\left(k - \frac{1}{2}\right)}{(4\pi m)^{k - \frac{1}{2}}}}{\frac{\Gamma\left(k - 1\right)}{(4\pi m)^{k - 1}}} 
\asymp
\frac{k^{1/2}}{m^{1/2}}
$$
we get
\begin{equation*}
\frac{1}{\BRAK{P_{k,m},P_{k,m}}}
\int_0^\infty
\left(
\widehat{\phi}[0](y)\cdot \widehat{P_{k,m}}[m](y)
\right) e^{-2\pi m y}y^{k-2} \mathrm{d}y 
\ll_\epsilon
 (1 +  |t|)^\epsilon  \frac{k^{\frac{1}{2}}}{m^{\frac{1}{2}}}.
\end{equation*}
Combining this with \eqref{eq-first_part_eisenstein_approx} finishes the proof.
\end{proof}

\subsection{Proof of mass equidistribution}
We now proceed with the proof of \autoref{thm-mass_equi}.

We will use the notation
$$
\widetilde{P}_{k,m} = \frac{P_{k,m}}{\BRAK{P_{k,m}, P_{k,m}}^{\frac{1}{2}}}.
$$
Let $\psi\in C_c^\infty\left(\SL\backslash\HH\right)$ be a smooth compactly supported function.
We wish to show that 
$$
\BRAK{\psi, |\widetilde{P}_{k,m}|^2}
\xrightarrow{k\rightarrow\infty}
\frac{3}{\pi}\BRAK{\psi, 1}.
$$

For $\psi$ as above we have (see for example \cite[Eq. 2.11]{arXiv:2405.00996})
\begin{equation} \label{eq-fast_deacy_psi}
\BRAK{\psi, \phi_j} \ll_{A,\psi}(1 + |t_j|)^{-A},
\qquad
\BRAK{\psi, E\left(\cdot , \frac{1}{2} + it\right)} \ll_{A,\psi} (1 + |t|)^{-A}
\end{equation}
for all $A\geq 1$. 
Here $\phi_j$ is a Maass cusp form with associated eigenvalue $\frac{1}{4} + t_j^2$.

From the spectral decomposition of $L^2\left(\SL \backslash \HH\right)$ (see e.g. \cite[Theorem 15.5]{MR2061214}) and Parseval's identity, we have
\begin{multline}\label{eq-parseval}
\BRAK{\psi,\left|\widetilde{P}_{k,m}^2\right|}
=
\frac{3}{\pi}\BRAK{\psi,1}
+ \sum_{j\geq 1}\BRAK{\left|\widetilde{P}_{k,m}^2\right|, \phi_j}\BRAK{\psi, \phi_j} \\
+ \int_{-\infty}^{\infty}
\BRAK{\left|\widetilde{P}_{k,m}^2\right|,E\left(\cdot\;, \frac{1}{2} + it\right)}
\BRAK{\psi,E\left(\cdot\;, \frac{1}{2} + it\right)} 
\mathrm{d} t.
\end{multline}
Here $\phi_j$ goes over a normalized basis for the cuspidial subspace 
consisting of Hecke Maass cusp forms, with associated eigenvalue $\frac{1}{4} + t_j^2$, such that $0<t_1 \leq t_2 \leq ...$.

And so, in order to prove
$$
\BRAK{\psi, |\widetilde{P}_{k,m}|^2}
\xrightarrow{k\rightarrow\infty}
\frac{3}{\pi}\BRAK{\psi, 1},
$$
it is enough to show that
$$
S = \sum_{j\geq 1}\BRAK{\left|\widetilde{P}_{k,m}^2\right|, \phi_j}\BRAK{\psi, \phi_j}
$$
and 
$$
I=\int_{-\infty}^{\infty}
\BRAK{\left|\widetilde{P}_{k,m}^2\right|,E\left(\cdot\;, \frac{1}{2} + it\right)}
\BRAK{\psi,E\left(\cdot\;, \frac{1}{2} + it\right)} 
\mathrm{d} t
$$
are arbitrarily small for large enough $k$.

We begin with $S$.
We pick a parameter $T>1$ and we split $S$ as 
$$
S = S_{\leq T} + S_{>T} =
\sum_{t_j\leq T}\BRAK{\left|\widetilde{P}_{k,m}^2\right|, \phi_j}\BRAK{\psi, \phi_j}
+
\sum_{t_j> T}\BRAK{\left|\widetilde{P}_{k,m}^2\right|, \phi_j}\BRAK{\psi, \phi_j}.
$$
In order to bound $S_{>T}$ we use
$$
\BRAK{\left|\widetilde{P}_{k,m}^2\right|, \phi_j}
\leq
\|\phi_j\|_{\infty}\BRAK{\widetilde{P}_{k,m},\widetilde{P}_{k,m}}=\|\phi_j\|_{\infty}.
$$
We use the bound $\|\phi_j\|_{\infty} \ll t_j^{\frac{1}{2}}$ (see \cite[Eq. 8.3']{MR1942691}).
We note that the stronger bound $\|\phi_j\|_{\infty} \ll_\epsilon t_j^{\frac{5}{12} + \epsilon}$ \cite[Appendix 1]{MR1324136} is also available, but is not needed.
Together with \eqref{eq-fast_deacy_psi} we get
$$
S_{> T} \ll_{A,\psi} \sum_{t_j > T} t_j^{-A + \frac{1}{2}}.
$$
Using Selberg's Weyl type law
$$
\#\SET{\phi_j \; : \; t_j \leq x} \ll x^2
$$
and partial summation, this gives
$$
S_{> T} \ll_{A,\psi} T^{-A + \frac{5}{2}}.
$$
And so, this sum can be made arbitrarily small by choosing $T$ large enough.

After fixing such a $T$, we now show that $S_{\leq T}$ and $I$ become arbitrarily small for sufficiently large $k$.
From \autoref{prop-inner_prod_maass} we have that for all $j$ we have 
$$
\BRAK{\left|\widetilde{P}_{k,m}^2\right|, \phi_j} 
\xrightarrow{k\rightarrow\infty} 0.
$$
And so, it follows that the finite sum $S_{\leq T}$ also satisfies $S_{\leq T}\xrightarrow{k\rightarrow\infty} 0$.

As for $I$, from \autoref{prop-inner_prod_eisenstein} we have that 
$$
\BRAK{\left|\widetilde{P}_{k,m}^2\right|, E\left(\cdot\;, \frac{1}{2} + it\right)} \ll_\epsilon (1 + |t|)^{\epsilon} o(1).
$$
Using \eqref{eq-fast_deacy_psi} we have
$$
I \ll_{\epsilon,A,\psi} \int_{-\infty}^{\infty}(1 + |t|)^{-A + \epsilon} \mathrm{d} t \cdot o(1)
\ll o(1).
$$
And so, once more we have $I\rightarrow 0$ as $k\rightarrow\infty$.

It follows from \eqref{eq-parseval} that
$$
\left|\BRAK{\psi,\left|\widetilde{P}_{k,m}^2\right|}
-
\frac{3}{\pi}\BRAK{\psi,1}\right|
$$
becomes arbitrarily small as $k\rightarrow \infty$, which finishes the proof.

\subsection{Uniform distribution of zeros}\label{subsec-zeros}

A non-zero modular form $f\in \MM_k$ has roughly $\frac{k}{12}$ zeros in $\mathcal{F}$.
More precisely, such a form satisfies the valence formula:
\begin{equation*}
v_\infty(f) + \sum_{p \in \mathcal{F}}w(p) v_p(f)
= \frac{k}{12}
\end{equation*}
where $v_p(f)$ is the order of vanishing of $f$ at $p$, and
$$
w(p) = \begin{cases}
   \frac{1}{2}  & p = i \\ 
   \frac{1}{3}  & p = \frac{1}{2} + \frac{\sqrt{3}}{2}i \\
   1            & \text{otherwise}
\end{cases}.
$$

In \cite[Theorem 2]{MR2181743} Rudnick showed that mass equidistribution implies equidistribution of zeros:
\begin{theorem}[Rudnick]\label{thm-equidist_zeros}
Let $\{f_k\}$ be a sequence of $L^2$ normalized cusp forms $f_k\in \SS_k$, for which the bulk of the zeros lie in the fundamental domain: $v_\infty(f_k) = o(k)$.
If $\mathrm{d}\mu_f  \xrightarrow{wk-*} \frac{3}{\pi}\mathrm{d}\mu $, then the zeros of $f_k$ become equidistributed in $\mathcal{F}$ with respect to the normalized hyperbolic measure $\frac{3}{\pi}\mathrm{d}\mu$ in the following sense:
for any nice compact subset $\Omega \subset \mathcal{F}$, we have
$$
\frac{12}{k} \sum_{p \in \Omega}w(p)v_{p}(f_k)
\sim \frac{3}{\pi }\int_\Omega \mathrm{d}\mu.
$$
\end{theorem}
Here $\xrightarrow{wk-*}$ denotes weak convergence when we test against compactly supported functions.

\begin{remark}
A careful examination of the proof of \cite[Theorem 2]{MR2181743} shows that the condition $v_\infty(f_k) = o(k)$ is not necessary.
In fact, the proof in \cite{MR2181743} actually shows that any sequence of cusp forms $\{f_k\}$ satisfying holomorphic QUE must satisfy $v_\infty(f_k) = o(k)$ as well.
We outline how this follows from the original proof of \cite[Theorem 2]{MR2181743}.

In equation 3.2, one does not need to divide by $\nu(f_k)$.
Proceeding with the proof without this division then shows that for any nice subset $\Omega \subset \mathcal{F}$, the number of zeros of $f_k$ in $\Omega$ is $\frac{k}{12}\cdot \frac{3}{\pi}\cdot \mathrm{Vol}\left(\Omega\right) + o(k)$.
And so, if one takes a very large $\Omega$, say of volume $\frac{\pi}{3} (1- \epsilon)$, then the result above shows that the number of zeros of $f_k$ in $\Omega$ is $\frac{k}{12}\left(1 -\epsilon\right) + o(k)$.
This accounts for most of the zeros of $f_k$, and so the number of zeros at the cusp is at most $ \frac{k}{12}\epsilon + o(k)$.
This is true for any $\epsilon>0$, which proves that $v_\infty(f_k) = o(k)$.
\end{remark}

\begin{remark}
In \cite{kimmel2024vanishing} we show that in fact  $P_{k,m}$ vanishes to order one at infinity   ($v_\infty(P_{k,m})=1$), if $m\leq \left( \frac{k-1}{4\pi}\right)^2$.
\end{remark}

We can use this in order to prove \autoref{cor-zeros}:
\begin{proof}
Let $\{P_{k,m}\}$ be a sequence satisfying the conditions of \autoref{cor-zeros}.
From \autoref{thm-mass_equi} we have that $\{P_{k,m}\}$ has mass equidistribution.
And so we can apply \autoref{thm-equidist_zeros} to $\{P_{k,m}\}$ and the result follows.
\end{proof}

\section*{Acknowledgments}
I would like to thank my Ph.D. advisor Zeev Rudnick for his guidance, insightful discussions, and for introducing me to this subject.
I would also like to thank Bingrong Huang for his comments on an earlier draft.
Lastly I would like to thank the anonymous referee for their suggestions and corrections.

\bibliographystyle{plain}
\bibliography{my_bib}

\end{document}